\documentclass[12pt,reqno]{amsart}
\usepackage{amsmath, amsfonts, amssymb, amsthm, hyperref}
\usepackage{bm}
\allowdisplaybreaks[4]
\def\l{\left}
\def\r{\right}
\def\bg{\bigg}
\def\({\bg(}
\def\){\bg)}
\def\t{\text}
\def\f{\frac}

\def\eq{\equiv}

\def\1{{\bf 1}}

\theoremstyle{plain}
\newtheorem{theorem}{Theorem}[section]
\newtheorem{lemma}{Lemma}

\theoremstyle{definition}

\theoremstyle{remark}
\newtheorem{remark}{Remark}

\def\<{\langle}
\def\>{\rangle}

\begin{document}
\hbox{}
\medskip

\title[Symbolic summation methods and hypergeometric supercongruences]{Symbolic summation methods and hypergeometric supercongruences}

\author{Chen Wang}
\address {(Chen Wang) Department of Mathematics, Nanjing
University, Nanjing 210093, People's Republic of China}
\email{cwang@smail.nju.edu.cn}

\thanks{This work is supported by the National Natural Science Foundation of China (grant no. 11971222)}
\begin{abstract} In this paper, we establish the following two congruences:
\begin{gather*}
\sum_{k=0}^{(p+1)/2}(3k-1)\frac{\left(-\frac{1}{2}\right)_k^2\left(\frac{1}{2}\right)_k4^k}{k!^3}\equiv p-6p^3\left(\frac{-1}{p}\right)+2p^3\left(\frac{-1}{p}\right)E_{p-3}\pmod{p^4},\\
\sum_{k=0}^{p-1}(3k-1)\frac{\left(-\frac{1}{2}\right)_k^2\left(\frac{1}{2}\right)_k4^k}{k!^3}\equiv p-2p^3\pmod{p^4},
\end{gather*}
where $p>3$ is a prime, $E_{p-3}$ is the $(p-3)$-th Euler number and $\left(-\right)$ is the Legendre symbol. The first congruence modulo $p^3$ was conjectured by Guo and Schlosser recently.
\end{abstract}
\maketitle

\section{Introduction}
\setcounter{lemma}{0}
\setcounter{theorem}{0}
\setcounter{equation}{0}
\setcounter{conjecture}{0}
\setcounter{remark}{0}

In 1997, Van Hamme \cite{vH} conjectured that for any odd prime $p$ one have
\begin{equation}\label{vhcon}
\sum_{k=0}^{(p-1)/2}(-1)^k(4k+1)\f{\l(\f12\r)_k^3}{k!^3}\eq p\l(\f{-1}{p}\r)\pmod{p^3},
\end{equation}
where $(x)_k=x(x+1)\cdots(x+k-1)$ is the Pochhammer symbol and $\l(-\r)$ is the Legendre symbol. This is a $p$-adic analogue of the following Ramanujan-type formula for $1/\pi$ due to Bauer \cite{Ba}:
$$
\sum_{k=0}^{\infty}(-1)^k(4k+1)\f{\l(\f12\r)_k^3}{k!^3}=\f{2}{\pi}.
$$
\eqref{vhcon} was later confirmed by Mortenson \cite{Mo} in 2008. In 2018, Guo \cite{Guo18} gave a $q$-analogue of \eqref{vhcon} as follows:
$$
\sum_{k=0}^{(p-1)/2}(-1)^kq^{k^2}[4k+1]\f{(q;q^2)_k^3}{(q^2;q^2)_k^3}\eq [p]q^{(p-1)^2/4}(-1)^{(p-1)/2}\pmod{[p]^3},
$$
where $(x;q)_n=(1-x)(1-xq)\cdots(1-xq^{n-1})$ denotes the $q$-Pochhammer symbol and $[n]=1+q+\cdots+q^{n-1}$ denotes the $q$-integer. In 2012, Sun obtained the following refinement of \eqref{vhcon} modulo $p^4$:
\begin{equation}\label{sunrefine}
\sum_{k=0}^{(p-1)/2}(-1)^k(4k+1)\f{\l(\f12\r)_k^3}{k!^3}\eq p\l(\f{-1}{p}\r)+p^3E_{p-3}\pmod{p^4},
\end{equation}
where $E_n$ is the $n$-th Euler number defined by
$$
\sum_{n=0}^{\infty}E_n\f{x^n}{n!}=\f{2e^x}{e^{2x}+1}\quad\t{for}\quad|x|<\f{\pi}{2}.
$$

Recently, Guo, Schlosser and Zudilin studied the $q$-congruences concerning basic hypergeometric series systematically. For example, in 2019, Guo and Zudilin \cite{GuoZu} developed the so-called $q$-microscope method to prove a series of basic hypergeometric congruences. For more details about their work one may consult \cite{GuoSch19a,GuoSch19b,GuoSch19c,GuoSch19d,GuoZu}.

In 2011, Sun \cite[Conj. 5.1(ii)]{Sun11} conjectured that for any prime $p>3$ one have
\begin{equation}\label{sun11conj}
\sum_{k=0}^{(p-1)/2}(3k+1)\f{\l(\f12\r)_k^34^k}{k!^3}\eq p+2p^3\l(\f{-1}{p}\r)E_{p-3}\pmod{p^4}.
\end{equation}
This was confirmed by Mao and Zhang in \cite{MZ}. In 2019, Guo and Schlosser \cite[(6.5)]{GuoSch19c} obtained a $q$-analogue of \eqref{sun11conj} modulo $p^3$. Analogously, they proposed the following conjectural congruence:
\begin{equation}\label{GuoSchconj}
\sum_{k=0}^{(p+1)/2}(3k-1)\frac{\left(-\frac{1}{2}\right)_k^2\left(\frac{1}{2}\right)_k4^k}{k!^3}\equiv p\pmod{p^3},
\end{equation}
where $p$ is an odd prime.

The main goal of this paper is to show \eqref{GuoSchconj} by establishing the following generalization.
\begin{theorem}\label{mainth1}
For any prime $p>3$, we have
\begin{equation}\label{mainth1eq}
\sum_{k=0}^{(p+1)/2}(3k-1)\frac{\left(-\frac{1}{2}\right)_k^2\left(\frac{1}{2}\right)_k4^k}{k!^3}\equiv p-6p^3\left(\frac{-1}{p}\right)+2p^3\left(\frac{-1}{p}\right)E_{p-3}\pmod{p^4}.
\end{equation}
\end{theorem}
\begin{remark}
One may check \eqref{GuoSchconj} for $p=3$ directly. Thus \eqref{mainth1eq} is actually an extension of \eqref{GuoSchconj}.
\end{remark}

Note that for $k\geq(p+3)/2$ we have
$$
\left(-\frac{1}{2}\right)_k^2\left(\frac{1}{2}\right)_k\eq0\pmod{p^3}.
$$
Thus we have
$$
\sum_{k=0}^{p-1}(3k-1)\frac{\left(-\frac{1}{2}\right)_k^2\left(\frac{1}{2}\right)_k4^k}{k!^3}\eq\sum_{k=0}^{(p+1)/2}(3k-1)\frac{\left(-\frac{1}{2}\right)_k^2\left(\frac{1}{2}\right)_k4^k}{k!^3}\pmod{p^3}.
$$
However, the above congruence does not hold modulo $p^4$. Below we state our second result.
\begin{theorem}\label{mainth2}
For any prime $p>3$, we have
\begin{equation}\label{mainth2eq}
\sum_{k=0}^{p-1}(3k-1)\frac{\left(-\frac{1}{2}\right)_k^2\left(\frac{1}{2}\right)_k4^k}{k!^3}\equiv p-2p^3\pmod{p^4}.
\end{equation}
\end{theorem}

We shall prove these two theorem by WZ method and Mathematica Package \verb"Sigma". One may refer to \cite{PWZ} and \cite{S} for the usage of these tools respectively.
\medskip
\section{Proofs of Theorems \ref{mainth1} and \ref{mainth2}}
\setcounter{lemma}{0}
\setcounter{theorem}{0}
\setcounter{equation}{0}
\setcounter{conjecture}{0}
\setcounter{remark}{0}

Our proofs are based on the following Wilf-Zeilberger pair (WZ pair) which can be verified directly.
\begin{lemma}\label{wzpair}
Set
$$
F(n,k)=(6n^2-5n+1-4nk+2k)\f{\l(-\f12-k\r)_n^2\l(-\f12\r)_n4^n}{n!^3}
$$
and
$$
G(n,k)=\f{4(-2n+1)n^3}{(3+2k-2n)^2}\cdot\f{\l(-\f12-k\r)_n^2\l(-\f12\r)_n4^n}{n!^3}.
$$
Then we have
$$
F(n,k+1)-F(n,k)=G(n+1,k)-G(n,k)
$$
for any nonnegative integers $n$ and $k$.
\end{lemma}
\begin{remark}
Finding a new WZ pair is not easy. We have never seen the above WZ pair in any former literature.
\end{remark}

\begin{lemma}\label{sigmaid}
Let $n$ be a positive integer. Then we have
\begin{gather*}
\sum_{k=1}^n\f{(-n)_k(1+n)_k}{(1)_k^2}H_k^{(2)}=-2(-1)^n\sum_{k=1}^{n}\f{(-1)^n}{k^2},\\
\sum_{k=1}^n\f{(-n)_k(1+n)_k}{(1)_k^2}H_k^2=4(-1)^nH_n^2+2(-1)^n\sum_{k=1}^{n}\f{(-1)^n}{k^2},\\
\sum_{k=1}^n\f{(-n)_k(1+n)_k}{(1)_k^2}H_k=2(-1)^nH_n,
\end{gather*}
where $H_n^{m}=\sum_{k=1}^n1/k^m$ is the $n$-th harmonic number of order $m$.
\end{lemma}
\begin{proof}
These identities were found by \verb"Sigma". Here we just prove the third one as an example. Set
$$
S_n :=\sum_{k=1}^n\f{(-n)_k(1+n)_k}{(1)_k^2}H_k.
$$
\par {\bf Step 1:} Load \verb"Sigma" in Mathematica and input $S_n$;
\par {\bf Step 2:} Use the command \verb"GenerateRecurrence" to find that $S_n$ satisfies
$$
(1+n)S_n+(3+2n)S_{n+1}+(n+2)S_{n+2}=0;
$$
\par {\bf Step 3:} Use the command \verb"SolveRecurrence" to solve the above recurrence relation and obtain a particular solution $\{0\}$
and the basic system of solutions $\{-(-1)^n, (-1)^nH_n\}$ of the homogeneous version;
\par {\bf Step 4:} Use the command \verb"FindLinearCombination" to get another form of $S_n$ as follows
$$
S_n=2(-1)^nH_n.
$$
\end{proof}

\begin{remark}
The identities in Lemma \ref{sigmaid} may appeared in \cite{G}.
\end{remark}

\begin{lemma}\cite{Sunzh1,Sunzh2}\label{sunzhlemma} For any prime $p>3$ we have
\begin{gather*}
H_{p-1}\eq0\pmod{p^2},\quad H_{(p-1)/2}\eq-2q_p(2)+pq_p^2(2)\pmod{p^2},\\
H_{p-1}^{(2)}\eq H_{(p-1)/2}^{(2)}\eq0\pmod{p},\quad H_{\lfloor p/4\rfloor}\eq4(-1)^{(p-1)/2}E_{p-3}\pmod{p},
\end{gather*}
where $q_p(2)=(2^{p-1}-1)/p$ denotes the Fermat quotient.
\end{lemma}

\begin{lemma}\label{-1k^2} For any prime $p>3$, we have
$$
\sum_{k=1}^{(p-1)/2}\f{(-1)^k}{k^2}\eq2(-1)^{(p-1)/2}E_{p-3}\pmod{p}.
$$
\end{lemma}
\begin{proof} By Lemma \ref{sunzhlemma},
\begin{align*}
\sum_{k=1}^{(p-1)/2}\f{(-1)^k}{k^2}\eq\sum_{k=1}^{(p-1)/2}\f{1+(-1)^k}{k^2}=\sum_{\substack{k=1\\ 2\mid k}}^{(p-1)/2}\f{2}{k^2}\f{1}{2}\sum_{k=1}^{\lfloor p/4\rfloor}\f{1}{k^2}\eq2(-1)^{(p-1)/2}E_{p-3}\pmod{p}.
\end{align*}
This proves Lemma \ref{-1k^2}.
\end{proof}

Recall that the Bernoulli numbers $B_0,B_1,\ldots$ are defined by
$$
\f{x}{e^x-1}=\sum_{n=0}^{\infty}B_n\f{x^n}{n!}\quad(0<|x|<2\pi).
$$
The Bernoulli polynomials are given by
$$
B_n(x)=\sum_{k=0}^n\binom{n}{k}B_kx^{n-k}\quad(n=0,1,2,\ldots).
$$
The congruences in the following lemma have been proved by different authors.
\begin{lemma}\label{2kk} Let $p>3$ be a prime. Then
\begin{gather*}
\sum_{k=0}^{(p-1)/2}\f{\binom{2k}{k}^2}{16^k}\eq(-1)^{(p-1)/2}+p^2E_{p-3}\pmod{p^3},\quad \t{(Sun \cite{Sun11})}\\
\sum_{k=1}^{p-1}\f{\binom{2k}{k}}{k}\eq0\pmod{p^2}\pmod{p^2},\quad \t{(Sun and Tauraso \cite{ST})}\\
\sum_{k=1}^{(p-1)/2}\f{\binom{2k}{k}}{k}\eq(-1)^{(p+1)/2}\f{8}{3}pE_{p-3}\pmod{p^2},\quad \t{(Sun \cite{Sun11})}\\
\sum_{k=1}^{p-1}\f{\binom{2k}{k}}{k^2}\eq\f{1}{2}\l(\f{p}{3}\r)B_{p-2}\l(\f13\r)\pmod{p},\quad \t{(Mattarei and Tauraso \cite{MT})}\\
\sum_{k=1}^{p-1}\f{\binom{2k}{k}}{k}H_k\eq\f{1}{3}\l(\f{p}{3}\r)B_{p-2}\l(\f13\r)\pmod{p}.\quad \t{(Mao and Sun \cite{MS})}
\end{gather*}
\end{lemma}

\noindent{\it Proof of Theorem \ref{mainth1}}. It is easy to see that
$$
(3n-1)\f{\l(-\f12\r)_n^2\l(\f12\r)_n4^n}{n!^3}=-(6n^2-5n+1)\f{\l(-\f12\r)_n^34^n}{n!^3}=F(n,0),
$$
where $F(n,k)$ is defined in Lemma \ref{wzpair}. Thus by Lemma \ref{wzpair} we have
\begin{align}\label{key}
&\sum_{n=0}^{(p+1)/2}(3n-1)\f{\l(-\f12\r)_n^2\l(\f12\r)_n4^n}{n!^3}=-\sum_{n=0}^{(p+1)/2}F(n,0)=-F\l(\f{p+1}{2},0\r)-\sum_{n=0}^{(p-1)/2}F(n,0)\notag\\
=&-F\l(\f{p+1}{2},0\r)-\sum_{n=0}^{(p-1)/2}\sum_{k=0}^{(p-3)/2}(F(n,k)-F(n,k+1))-\sum_{n=0}^{(p-1)/2}F\l(n,\f{p-1}{2}\r)\notag\\
=&-F\l(\f{p+1}{2},0\r)-\sum_{n=0}^{(p-1)/2}\sum_{k=0}^{(p-3)/2}(G(n,k)-G(n+1,k))-\sum_{n=0}^{(p-1)/2}F\l(n,\f{p-1}{2}\r)\notag\\
=&-F\l(\f{p+1}{2},0\r)+\sum_{k=0}^{(p-3)/2}G\l(\f{p+1}{2},k\r)-\sum_{n=0}^{(p-1)/2}F\l(n,\f{p-1}{2}\r).
\end{align}

Below we first consider $$F\l(\f{p+1}{2},0\r)\pmod{p^4}$$. It is routine to check that
$$
\l(6\l(\f{p+1}{2}\r)^2-5\l(\f{p+1}{2}\r)+1\r)\bigg/(p+1)^3\eq \f{1}{2}p-\f{3}{2}p^3\pmod{p^4}.
$$
Thus we have
\begin{align*}
F\l(\f{p+1}{2},0\r)\eq -\l(2p-6p^3\r)4^{(p-1)/2}\f{\l(\f12\r)_{(p-1)/2}^3}{(1)_{(p-1)/2}^3}=-\l(2p-6p^3\r)\f{\binom{p-1}{(p-1)/2}^3}{4^{p-1}}\pmod{p^3}.
\end{align*}
Recall Morley's congruence (cf. \cite{Mor})
\begin{equation}\label{morley}
\binom{p-1}{\f{p-1}{2}}\eq(-1)^{(p-1)/2}4^{p-1}\pmod{p^3}.
\end{equation}
Thus we have
\begin{align}\label{Fp/20}
F\l(\f{p+1}{2},0\r)\eq&-\l(2p-6p^3\r)(-1)^{(p-1)/2}16^{p-1}\eq-\l(2p-6p^3\r)(-1)^{(p-1)/2}(1+pq_p(2))^4\notag\\
\eq&(-1)^{(p+1)/2}(2p-6p^3+8p^2q_p(2)+12p^3q_p^2(2))\pmod{p^4}.
\end{align}

Now we consider
$$\sum_{k=0}^{(p-3)/2}G\l(\f{p+1}{2},k\r)\pmod{p^4}.$$
By the definition of $G(n,k)$ we have
\begin{align*}
\sum_{k=0}^{(p-3)/2}G\l(\f{p+1}{2},k\r)=-p\sum_{k=0}^{(p-3)/2}\f{\l(-\f12-k\r)_{(p-1)/2}^2\l(-\f12\r)_{(p+1)/2}4^{(p+1)/2}}{(1)_{(p-1)/2}^3}.
\end{align*}
It is easy to check that
$$
\l(-\f12-k\r)_{(p-1)/2}=(2-p)\l(-\f12\r)_{(p-1)/2}\f{\l(\f12\r)_{k+1}}{\l(1-\f p2\r)_{k+1}}.
$$
Hence we obtain
\begin{equation}\label{Gp2k1}
\sum_{k=0}^{(p-3)/2}G\l(\f{p+1}{2},k\r)=-p(p-2)^2\f{\l(-\f12\r)_{(p-1)/2}^2\l(-\f12\r)_{(p+1)/2}4^{(p+1)/2}}{(1)_{(p-1)/2}^3}\sum_{k=0}^{(p-3)/2}\f{\l(\f12\r)_{k+1}^2}{\l(1-\f p2\r)_{k+1}^2}.
\end{equation}
Now by Morley's congruence \eqref{morley} again we get
\begin{align}\label{Gp2k2}
&(p-2)^2\f{\l(-\f12\r)_{(p-1)/2}^2\l(-\f12\r)_{(p+1)/2}4^{(p+1)/2}}{(1)_{(p-1)/2}^3}=-2\cdot4^{(p-1)/2}\cdot\f{\l(\f{1}{2}\r)_{(p-1)/2}^3}{(1)_{(p-1)/2}^3}=-2\f{\binom{p-1}{(p-1)/2}^3}{4^{p-1}}\notag\\
\eq&2(-1)^{(p+1)/2}16^{p-1}=2(-1)^{(p+1)/2}(1+4pq_p(2)+6p^2q_p^2(2))\pmod{p^3}.
\end{align}
On the other hand, we have
\begin{align}\label{Gp2k3}
&\sum_{k=0}^{(p-3)/2}\f{\l(\f12\r)_{k+1}^2}{\l(1-\f p2\r)_{k+1}^2}=\sum_{k=1}^{(p-1)/2}\f{\l(\f12\r)_{k}^2}{\l(1-\f p2\r)_{k}^2}\notag\\
\eq&\sum_{k=1}^{(p-1)/2}\f{\l(\f12\r)_{k}^2}{\l(1\r)_{k}^2(1-pH_k/2+p^2(H_k^2-H_k^{(2)})/8)^2}\notag\\
\eq&\sum_{k=1}^{(p-1)/2}\f{\l(\f12\r)_{k}^2}{\l(1\r)_{k}^2}\l(1+pH_k+\f{p^2}{2}H_k^2+\f{p^2}{4}H_k^{(2)}\r)\pmod{p^3}
\end{align}
by noting that
\begin{align*}
\l(1-\f{p}{2}\r)_k\eq&(1)_k\l(1-\f{p}{2}H_k+\f{p^2}{4}\sum_{1\leq i<j\leq k}\f{1}{ij}\r)\\
=&(1)_k\l(1-\f{p}{2}H_k+\f{p^2}{8}(H_k^2-H_k^{(2)})\r)\pmod{p^3}.
\end{align*}
Clearly, for $k=1,\ldots,(p-1)/2$ we have
$$
\l(\f12\r)_k^2\eq\l(\f{1-p}2\r)_k\l(\f{1+p}2\r)_k.
$$
Thus by Lemma \ref{sigmaid} we have
\begin{gather*}
\sum_{k=1}^{(p-1)/2}\f{\l(\f12\r)_{k}^2}{\l(1\r)_{k}^2}H_k\eq2(-1)^{(p-1)/2}H_{(p-1)/2}\pmod{p^2},\\
\sum_{k=1}^{(p-1)/2}\f{\l(\f12\r)_{k}^2}{\l(1\r)_{k}^2}H_k^2\eq4(-1)^{(p-1)/2}H_{(p-1)/2}^2+2(-1)^{(p-1)/2}\sum_{k=1}^{(p-1)/2}\f{(-1)^k}{k^2}\pmod{p},\\
\sum_{k=1}^{(p-1)/2}\f{\l(\f12\r)_{k}^2}{\l(1\r)_{k}^2}H_k^{(2)}\eq-2(-1)^{(p-1)/2}\sum_{k=1}^{(p-1)/2}\f{(-1)^k}{k^2}\pmod{p}.
\end{gather*}
Combining these with \eqref{Gp2k1}--\eqref{Gp2k3} and in view of Lemmas \ref{sunzhlemma}, \ref{-1k^2} and \ref{2kk} we arrive at
\begin{align}\label{Gp2k}
&\sum_{k=0}^{(p-3)/2}G\l(\f{p+1}{2},k\r)\notag\\
\eq&2p(-1)^{(p-1)/2}(1+4pq_p(2)+6p^2q_p^2(2))\notag\\
&\times((-1)^{(p-1)/2}+2p^2E_{p-3}-1-4(-1)^{(p-1)/2}pq_p(2)+10(-1)^{(p-1)/2}p^2q_p^2(2))\notag\\
\eq&2p-2(-1)^{(p-1)/2}p-8(-1)^{(p-1)/2}p^2q_p(2)-12(-1)^{(p-1)/2}p^3q_p^2(2)\notag\\
&+4(-1)^{(p-1)/2}p^3E_{p-3}\pmod{p^4}.
\end{align}

Finally, we consider
$$
\sum_{n=0}^{(p-1)/2}F\l(n,\f{p-1}{2}\r)\pmod{p^4}.
$$
By the definition of $F(n,k)$ we get that
\begin{align}\label{key1}
\sum_{n=0}^{(p-1)/2}F\l(n,\f{p-1}{2}\r)=&p+\sum_{n=1}^{(p-1)/2}(6n^2-3n-2np+p)\f{\l(-\f{p}{2}\r)_n^2\l(-\f12\r)_n4^n}{(1)_n^3}\notag\\
=&p-\f{p^2}{8}\sum_{n=1}^{(p-1)/2}(6n^2-3n-2np+p)\f{\l(1-\f{p}{2}\r)_n^2\l(\f12\r)_n4^n}{(1)_n^3\l(n-\f p2\r)^2\l(n-\f12\r)}\notag\\
\eq&p-\f{p^2}{4}\sum_{n=1}^{(p-1)/2}\f{\l(\f12\r)_n4^n}{(1)_nn^3}(1-pH_n)(n+p)(3n-p)\notag\\
\eq&p-\f{3}{4}p^2\sum_{n=1}^{(p-1)/2}\f{\l(\f12\r)_n4^n}{(1)_nn}-\f{p^3}{2}\sum_{n=1}^{(p-1)/2}\f{\l(\f12\r)_n4^n}{(1)_nn^2}\notag\\
&+\f{3}{4}p^3\sum_{n=1}^{(p-1)/2}\f{\l(\f12\r)_n4^nH_n}{(1)_nn}\pmod{p^4}.
\end{align}
Note that
$$
\f{\l(\f12\r)_n}{(1)_n}=\f{\binom{2n}{n}}{4^n}.
$$
Thus by Lemma \ref{2kk} we have
\begin{align}\label{Fnp/2}
\sum_{n=0}^{(p-1)/2}F\l(n,\f{p-1}{2}\r)\eq p+2(-1)^{(p-1)/2}p^3E_{p-3}\pmod{p^4}.
\end{align}
Substituting \eqref{Fp/20}, \eqref{Gp2k} and \eqref{Fnp/2} into \eqref{key} we arrive at
\begin{align*}
&\sum_{n=0}^{(p+1)/2}(3n-1)\f{\l(-\f12\r)_n^2\l(\f12\r)_n4^n}{n!^3}\\
\eq&(-1)^{(p-1)/2}(2p-6p^3+8p^2q_p(2)+12p^3q_p^2(2))+2p-2(-1)^{(p-1)/2}p-8(-1)^{(p-1)/2}p^2q_p(2)\\
&-12(-1)^{(p-1)/2}p^3q_p^2(2)+4(-1)^{(p-1)/2}p^3E_{p-3}-p-2(-1)^{(p-1)/2}p^3E_{p-3}\\
\eq&p-6(-1)^{(p-1)/2}p^3+2(-1)^{(p-1)/2}p^3E_{p-3}\pmod{p^4}.
\end{align*}
This coincides with \eqref{mainth1eq} since $(-1)^{(p-1)/2}=\l(\f{-1}{p}\r)$.

The proof of Theorem \ref{mainth1} is now complete.\qed

\medskip

\noindent{\it Proof of Theorem \ref{mainth2}}. As in the proof of Theorem \ref{mainth1}, by Lemma \ref{wzpair}, we obtain
\begin{align}\label{key2}
\sum_{n=0}^{p-1}(3n-1)\f{\l(-\f12\r)_n^2\l(\f12\r)_n4^n}{n!^3}=-\sum_{n=0}^{p-1}F(n,0)=\sum_{k=0}^{(p-3)/2}G(p,k)-\sum_{n=0}^{p-1}F\l(n,\f{p-1}{2}\r).
\end{align}
Via a similar discussion as in the computation of \eqref{key1} we arrive at
\begin{align*}
\sum_{n=0}^{p-1}F\l(n,\f{p-1}{2}\r)\eq&p-\f{3}{4}p^2\sum_{n=1}^{p-1}\f{\l(\f12\r)_n4^n}{(1)_nn}-\f{p^3}{2}\sum_{n=1}^{p-1}\f{\l(\f12\r)_n4^n}{(1)_nn^2}\notag\\
&+\f{3}{4}p^3\sum_{n=1}^{p-1}\f{\l(\f12\r)_n4^nH_n}{(1)_nn}\pmod{p^4}.
\end{align*}
Then by Lemma \ref{2kk} we immediately get that
\begin{equation}\label{sum2}
\sum_{n=0}^{p-1}F\l(n,\f{p-1}{2}\r)\eq p\pmod{p^4}.
\end{equation}

Below we evaluate $\sum_{k=0}^{(p-3)/2}G(p,k)$ modulo $p^4$. Note that
$$
\l(-\f12-k\r)_{p-1}=\f{\l(-\f12\r)_{p-1}\l(\f32\r)_k}{\l(\f52-p\r)_k}
$$
and $p\nmid(3/2-p)_k$ for $k=1,\ldots,(p-5)/2$. It is easy to see that
\begin{align}\label{Gpk}
\sum_{k=0}^{(p-3)/2}G(p,k)=&\sum_{k=0}^{(p-3)/2}(-2p+1)\l(p-\f32\r)\f{\l(-\f12-k\r)_{p-1}^2\l(-\f12\r)_{p-1}4^p}{(1)_{p-1}^3}\notag\\
=&\sum_{k=0}^{(p-5)/2}\f{-2p+1}{\l(p-\f32\r)^2}\f{\l(-\f12\r)_{p}^3\l(\f32\r)_k^24^p}{(1)_{p-1}^3\l(\f52-p\r)_k^2}+G\l(p,\f{p-3}{2}\r).
\end{align}
In light of Lemma \ref{sunzhlemma}, we obtain
\begin{align}\label{-12p}
\l(-\f12\r)_{p}=&-\f12\cdot\f12\cdots\l(p-\f32\r)=-\f{(2p-1)!}{2^{2p-1}(2p-1)(p-1)!}\notag\\
=&-p\f{(1+p)_{p-1}}{2^{2p-1}(2p-1)}\eq-p\f{(1)_{p-1}\l(1+pH_{p-1}+\f{p^2}{4}(H_{p-1}^2-H_{p-1}^{(2)}\r)}{2^{2p-1}(2p-1)}\notag\\
\eq&\f{-p(1)_{p-1}}{2^{2p-1}(2p-1)}\pmod{p^4}.
\end{align}
Therefore by Fermat's little theorem we have
\begin{align}\label{Gpk1}
&\sum_{k=0}^{(p-5)/2}\f{-2p+1}{\l(p-\f32\r)^2}\f{\l(-\f12\r)_{p}^3\l(\f32\r)_k^24^p}{(1)_{p-1}^3\l(\f52-p\r)_k^2}\notag\\
\eq&2p^3\sum_{k=0}^{(p-5)/2}\f{1}{(2k+3)^2}=2p^3\l(\sum_{k=0}^{(p-3)/2}\f{1}{(2k+1)^2}-1\r)=2p^3\l(H_{p-1}^{(2)}-\f14H_{(p-1)/2}^{(2)}-1\r)\notag\\
\eq&-2p^3\pmod{p^4}.
\end{align}
On the other hand, with the help of \eqref{-12p}, we arrive at
\begin{align}\label{Gpk2}
G\l(p,\f{p-3}{2}\r)=&(-2p+1)\f{\l(1-\f p2\r)_{p-1}^2\l(-\f12\r)_p4^p}{(1)_{p-1}^3}\notag\\
\eq&-p(-2p+1)\f{(1)_{p-1}^34^p}{(1)_{p-1}^32^{2p-1}(2p-1)}=2p\pmod{p^4}.
\end{align}
Substituting \eqref{Gpk1} and \eqref{Gpk2} into \eqref{Gpk} we have
\begin{equation}\label{sum1}
\sum_{k=0}^{(p-3)/2}G(p,k)\eq2p-2p^3\pmod{p^4}.
\end{equation}
Combining \eqref{sum1} and \eqref{sum2} we finally obtain
$$
\sum_{n=0}^{p-1}(3n-1)\f{\l(-\f12\r)_n^2\l(\f12\r)_n4^n}{n!^3}\eq p-2p^3\pmod{p^4}.
$$
Now the proof of Theorem \ref{mainth2} is complete.\qed

\end{document}